\documentclass{amsart}

\usepackage[latin1]{inputenc}
\usepackage{amsmath, amsthm, amssymb}
\usepackage[colorlinks]{hyperref}

\newtheorem{defin}{Definition}[section]

\newtheorem{theorem}[defin]{Theorem}

\newtheorem{lemma}[defin]{Lemma}

\newcommand{\R}{\mathbb{R}}

\DeclareMathOperator{\cone}{cone}

\DeclareMathOperator{\interior}{int}
\DeclareMathOperator{\cl}{cl}

\DeclareMathOperator{\Trace}{Trace}

\title{Rational factorizations of completely positive matrices}

\author{Mathieu Dutour Sikiri\'c}
\address{M.~Dutour Sikiri\'c, Rudjer Boskovi\'c Institute, Bijenicka
  54, 10000 Zagreb, Croatia}
\email{mathieu.dutour@gmail.com}

\author{Achill Sch\"urmann}
\address{A.~Sch\"urmann, Universit\"at Rostock, Institute of
  Mathematics, 18051 Rostock, Germany}
\email{achill.schuermann@uni-rostock.de}

\author{Frank Vallentin}
\address{F.~Vallentin, Mathematisches Institut, Universit\"at zu
  K\"oln, Weyertal~86--90, 50931 K\"oln, Germany}
\email{frank.vallentin@uni-koeln.de}

\date{February 4, 2017}

\subjclass{90C25}

\keywords{copositive programming, completely positive matrix, cp-factorization}

\begin{document}

\begin{abstract}
  In this note it is proved that every rational matrix which lies in
  the interior of the cone of completely positive matrices also has a
  rational cp-factorization.
\end{abstract}

\maketitle

\markboth{M.~Dutour Sikiri\'c, A.~Sch\"urmann, F.~Vallentin}{Rational
  factorizations of completely positive matrices}

\section{Introduction}

The cone of completely positive matrices is central to copositive
programming, see \cite{Duer2010a} and also to several topics in matrix
theory, see \cite{Berman2003a}.  However, so far, this cone is quite
mysterious, many basic questions about it are open. In
\cite{Berman2015a} Berman, D\"ur, and Shaked-Monderer ask:
\textit{Given a matrix $A \in \mathcal{CP}_n$ all of whose entries are
  integral, does $A$ always have a rational cp-factorization?}

The \emph{cone of completely positive matrices} is defined as the
convex cone spanned by symmetric rank-$1$-matrices $xx^{\sf T}$ where
$x$ lies in the nonnegative orthant $\mathbb{R}^n_{\geq 0}$:
\[
\mathcal{CP}_n = \cone\{xx^{\sf T} : x \in \mathbb{R}^n_{\geq 0}\}.
\]
A \emph{cp-factorization} of a matrix $A$ is a factorization of the form
\[
A = \sum_{i=1}^m \alpha_i x_i x_i^{\sf T} \quad \text{with } \alpha_i
\geq 0 \text{ and } x_i \in \mathbb{R}^n_{\geq 0}, \quad \text{for } i = 1,
\ldots, m.
\]
We talk about a \emph{rational cp-factorization} when the $\alpha_i$'s
are rational numbers and when the $x_i$'s are rational vectors. Of
course, in a rational cp-factorization we can assume that the $x_i$'s
are integral vectors.

In this note we prove the following theorem:

\begin{theorem}
\label{thm:main}
Every rational matrix which lies in the interior of the cone of
completely positive matrices has a rational cp-factorization.
\end{theorem}

So to fully answer the question of Berman, D\"ur, and Shaked-Monderer,
it remains to consider the boundary of $\mathcal{CP}_n$.

\section{Proof of Theorem~\ref{thm:main}}

For the proof we will need a classical result from simultaneous
Diophantine approximation, a theorem of Dirichlet, which we state
here. One can find a proof of Dirichlet's theorem for example in the
book \cite [Theorem 5.2.1]{Groetschel1988a} of Gr\"otschel, Lov\'asz,
and Schrijver.

\begin{theorem}
\label{thm:Dirichlet}
  Let $\alpha_1, \ldots, \alpha_n$ be real numbers and let
  $\varepsilon$ be a real number with $0 < \varepsilon < 1$.  Then there
  exist integers $p_1, \ldots, p_n$ and a natural number $q$ with
  $1 \leq q \leq \varepsilon^{-n}$ such that
\[
 \left| \alpha_i - \frac{p_i}{q} \right| \leq \frac{\varepsilon}{q}
\quad\text{for all } i = 1, \ldots, n.
\]
\end{theorem}

The next lemma collects standard, easy-to-prove facts about convex
cones. Let $E$ be a Euclidean space with inner product
$\langle \cdot, \cdot \rangle$.  Let $K \subseteq E$ be a \emph{proper
  convex cone}, which means that $K$ is closed, has a
nonempty interior, and satisfies $K \cap (-K) = \{0\}$.  Its
\emph{dual cone} is defined as
$K^* = \{y \in E : \langle x, y \rangle \geq 0 \text{ for all } x \in K\}$.

\begin{lemma}
  \label{lem:cone-facts}
Let $K \subseteq E$ be a proper convex cone. Then,
\begin{equation}
\label{eq:interior}
\interior(K) = \{x \in E : \langle x, y \rangle > 0 \text{ for
      all } y \in K^* \setminus \{0\}\},
\end{equation}
where $\interior(K)$ is the topological interior of $K$, and
\begin{equation}
\label{eq:closure}
K^* = (\cl(K))^*,
\end{equation}
where $\cl(K)$ is the topological closure of $K$.
\end{lemma}

We need some more notation: With $\mathcal{S}^n$ we denote the vector space
of symmetric matrices with $n$ rows and $n$ columns which is a
Euclidean space with inner product
$\langle A, B \rangle = \Trace(AB) = \sum_{i,j=1}^n A_{ij}B_{ij}$. The
\emph{cone of copositive matrices} is the dual cone of
$\mathcal{CP}_n$:
\[
\mathcal{COP}_n = \mathcal{CP}_n^* = \{B \in \mathcal{S}^n :
  \langle A, B \rangle \geq 0 \text{ for all } A \in \mathcal{CP}\}.
\]
Its interior equals
\[
\interior(\mathcal{COP}_n) = \{B \in \mathcal{S}^n : \langle B,
  xx^{\sf T} \rangle > 0 \text{ for all } x \in \mathbb{R}^n_{\geq 0}
  \setminus \{0\}\}.
\]
We also define the following rational subcone of $\mathcal{CP}_n$:
\[
\tilde{\mathcal{CP}}_n = \cone\{vv^{\sf T} : v \in \mathbb{Z}^n_{\geq 0}\}.
\]

We prepare the proof of the paper's main result by two lemmata which
might be useful facts themselves.

\begin{lemma}
\label{lem:interior}
The set 
\[
\mathcal{R} = \{B \in \mathcal{S}^n : \langle B, vv^{\sf T} \rangle \geq 1 \text{
  for all } v \in \mathbb{Z}^n_{\geq 0} \setminus \{0\}\},
\]
is contained in the interior of the cone of copositive matrices
$\mathcal{COP}_n$.
\end{lemma}

\begin{proof}
  Since the set of nonnegative rational vectors
  $\mathbb{Q}^n_{\geq 0}$ lies dense in the nonnegative orthant
  $\mathbb{R}_{\geq 0}^n$, we have the inclusion
  $\mathcal{R} \subseteq \mathcal{COP}_n$. Suppose for contradiction
  that the set on the left is not contained in
  $\interior(\mathcal{COP}_n)$: There is a matrix $B$ with
  $\langle B, vv^{\sf T} \rangle \geq 1$ for all
  $v \in \mathbb{Z}^n_{\geq 0} \setminus \{0\}$ and there is a nonzero
  vector $x \in \mathbb{R}^n_{\geq 0}$ with
  $\langle B, xx^{\sf T} \rangle = 0$.

  By induction on $n$ (and reordering if necessary) we may assume that
  all entries of $x$ are strictly positive, $x_i > 0$ for all
  $i = 1, \ldots, n$, since otherwise, we can reduce the situation to
  the case of smaller dimension by considering a suitable submatrix of
  $B$.

Hence, the vector $x$ lies in the interior of the nonnegative
orthant. Therefore, and because $B \in \mathcal{COP}_n$, we have for every
vector $y \in \R^n$ and $\varepsilon > 0$ sufficiently small the
inequality
\[
0 \leq \frac{1}{\varepsilon}(x + \varepsilon y)^{\sf T} B (x +
\varepsilon y) = 2x^{\mathsf T} B y + \varepsilon y^{\sf T} B y
\]
and similarly
\[
0 \leq \frac{1}{\varepsilon}(x - \varepsilon y)^{\sf T} B (x -
\varepsilon y) = - 2x^{\mathsf T} B y + \varepsilon y^{\sf T} B y
\]
From this, equality $x^{\mathsf T} B = 0$ follows. From this, we also
see that $B$ is positive semidefinite. This implies that
\[
(\alpha x + y)^{\sf T} B (\alpha x + y) = y^{\sf T} B y \quad \text{for } \alpha\in \R \text{ and } y\in \R^n.
\]

We apply Dirichlet's approximation theorem,
Theorem~\ref{thm:Dirichlet} to the vector $x$ and to
$\varepsilon \in (0,1)$. We obtain a vector $p = (p_1, \ldots, p_n)$
and a natural number $q$. Since $x_i > 0$ we may without loss of
generality assume that $p_i \geq 0$. Thus, by the assumption $B \in
\mathcal{R}$, we have $\langle B, pp^{\sf T} \rangle \geq 1$.

Define 
\[
y = q x - p \quad \text{where} \quad \|y\|_{\infty} \leq \varepsilon.
\]
Since $B$ is positive semidefinite, there is a constant $C$ such that
$y^{\sf T} B y \leq C \|y\|_{\infty}^2$ for all $y \in \mathbb{R}^n$.
Putting everything together we get
\[
1 \leq  \langle B, pp^{\sf T} \rangle = (qx - y)^{\sf T} B  (qx - y) =  y^{\sf
  T} B y \leq C \|y\|_{\infty}^2 \leq C \varepsilon^{2},
\]
which yields a contradiction for small enough values of $\varepsilon$. 
\end{proof}

\begin{lemma}
  Let $A$ be a completely positive matrix which lies in the interior
  of $\mathcal{CP}_n$ and let $\lambda$ be a sufficiently large positive real
  number. Then 
  the set
\[
\mathcal{P}(A,\lambda) = \{B \in \mathcal{S}^n : \langle A, B \rangle \leq
\lambda, \; \langle B, vv^{\sf T} \rangle \geq 1 \text{ for all } v \in
\mathbb{Z}^n_{\geq 0} \setminus \{0\}\}
\]
is a full-dimensional polytope.
\end{lemma}

\begin{proof}
  For sufficiently large $\lambda$ a sufficiently small ball around a
  suitable multiple of $A$ is contained in $P(A,\lambda)$, which shows
  that $P(A,\lambda)$ has full dimension.

\smallskip

  By the theorem of Minkowski and Weyl, see for example
  \cite[Corollary 7.1c]{Schrijver1986a}, polytopes are exactly bounded
  polyhedra. So it suffices to show that the
  set $\mathcal{P}(A,\lambda)$ is a bounded polyhedron.

\smallskip

First we show that $\mathcal{P}(A, \lambda)$ is bounded: For suppose
not. Then there is $B_0 \in \mathcal{P}(A, \lambda)$ and
$B_1 \in \mathcal{S}^n$, with $B_1 \neq 0$, so that the ray
$B_0 + \alpha B_1$, with $\alpha \geq 0$, lies completely in
$\mathcal{P}(A,\lambda)$. In particular
$\langle B_1, vv^{\sf T} \rangle \geq 0$ for all
$v \in \mathbb{Z}^n_{\geq 0}$. Hence, $B_1$ lies in the dual cone of
$\tilde{\mathcal{CP}}_n$.
On the other hand $\langle A, B_1 \rangle \leq 0$. Hence, by
Lemma~\ref{lem:cone-facts} \eqref{eq:interior},
$B_1 \not\in \mathcal{COP}_n \setminus \{0\}$, but by
Lemma~\ref{lem:cone-facts} \eqref{eq:closure},
\[
\tilde{\mathcal{CP}}_n^* =  (\cl(\tilde{\mathcal{CP}}_n))^* = \mathcal{CP}_n^* =
\mathcal{COP}_n,
\]
so $B_1 = 0$, yielding a contradiction.

\smallskip

Now we show that $\mathcal{P}(A, \lambda)$ is a polyhedron: For
suppose not. Then there is a sequence
$v_i \in \mathbb{Z}^n_{\geq 0} \setminus \{0\}$ of infinitely many
pairwise different nonzero lattice vectors so that there are
$B_i \in \mathcal{P}(A,\lambda)$ with
$\langle B_i, v_iv_i^{\sf T} \rangle = 1$. Since
$\mathcal{P}(A, \lambda)$ is compact, there exists a subsequence
$B_{i_j}$ which converges to $B^* \in \mathcal{P}(A, \lambda)$. Define
the sequence $u_{i_j} = v_{i_j}/\|v_{i_j}\|$ which lies in the compact
set $\mathbb{R}^n_{\geq 0} \cap S^{n-1}$ where $S^{n-1}$ denotes the
unit sphere. Hence there is a subsequence converging to $u^* \in
S^{n-1}$, in particular $u^* \neq
0$. Denote the indices of this subsequence with $k$, then
\[
1 = \langle B_k, v_k v_k^{\sf T} \rangle = \|v_k\|^2 \langle B_k, u_k u_k^{\sf T} \rangle.
\]
When $k$ tends to infinity, the squared norms $\|v_k\|^2$ tend to
infinity as well, since we use infinitely many pairwise different
lattice vectors and there exist only finitely many lattice vectors up
to some given norm. So $\langle B_k, u_k u_k^{\sf T} \rangle$ tends to
$\langle B^*, u^*(u^*)^{\sf T} \rangle = 0$, and by
Lemma~\ref{lem:interior} we obtain a contradiction.
\end{proof}

Now we prove the main result and finish the paper.

\begin {proof}[Proof of Theorem~\ref{thm:main}]
  Let $A$ be matrix having rational entries only and lying in the
  interior of the cone of completely positive matrices. Then
  $\mathcal{P}(A,\lambda)$ is a polytope according to the previous
  lemma.  We minimize the linear functional
  $B \mapsto \langle A, B\rangle$ over $\mathcal{P}(A,\lambda)$. The
  minimum is attained at one of the polytopes' vertices,
  $B^* \in \mathcal{P}(A,\lambda)$. Then we choose those lattice
  vectors $v_i \in \mathbb{Z}_{\geq 0}^n$, with $i = 1, \ldots, m$ for
  which equality $\langle B^*, v_iv_i^{\sf T} \rangle = 1$
  holds. Because of the minimality of $\langle A, B^* \rangle$ it
  follows 
\begin{equation}
\label{eq:cone}
A \in \cone\{v_iv_i^{\sf T} : i = 1, \ldots, m\}.
\end{equation}
Otherwise, see for example \cite[Theorem 7.1]{Schrijver1986a}, we
find a separating linear hyperplane orthogonal to $C$ separating
$A$ and $\cone\{v_iv_i^{\sf T} : i = 1, \ldots, m\}$:
\[
\langle C, A \rangle < 0 \quad \text{and} \quad \langle C, v_iv_i^{\sf
  T} \rangle \geq 0 \; \text{ for all } i = 1, \ldots, m.
\]
Then for sufficiently small $\mu > 0$ we would have
\[
B^* + \mu C \in \mathcal{P}(A,\lambda) \quad \text{but} \quad
\langle B^* + \mu C, A \rangle < \langle B^*, A \rangle,
\]
which contradicts the minimality of $\langle A, B^* \rangle$.

We apply Carath\'eodory's theorem (see for example \cite[Corollary
7.1i]{Schrijver1986a}) to \eqref{eq:cone} and choose a subset
$I \subseteq \{1, \ldots, m\}$ so that $v_iv_i^{\sf T}$ are linearly
independent and so that $A$ lies in
$\cone\{v_iv_i^{\sf T} : i \in I\}$. Since $A$ is a rational matrix
and since the $v_iv_i^{\sf T}$'s are linearly independent rational
matrices, there is a unique choice of rational numbers
$\alpha_i \in \mathbb{Q}_{\geq 0}$, with $i \in I$, so that
$A = \sum_{i \in I} \alpha_i v_i v_i^{\sf T}$ holds, which gives a
desired rational cp-factorization.
\end{proof}

\section*{Acknowledgements}

We thank Naomi Shaked-Monderer for comments on the manuscript.

\end{document}